\documentclass[journal]{IEEEtran}

\usepackage{stfloats}
\usepackage{amsthm}
\usepackage{aliascnt}
\usepackage{multirow}
\usepackage{flushend}
\usepackage{mathtools}
\usepackage{siunitx}
\usepackage{colortbl}
\usepackage[noadjust]{cite}
\theoremstyle{plain}
\newtheorem{theorem}{Theorem}[section]
\newaliascnt{proposition}{theorem}

\aliascntresetthe{proposition}

\newaliascnt{corollary}{theorem}
\newtheorem{corollary}[corollary]{Corollary}
\aliascntresetthe{corollary}

\newaliascnt{lemma}{theorem}
\newtheorem{lemma}[lemma]{Lemma}
\aliascntresetthe{lemma}

\theoremstyle{definition}
\newaliascnt{definition}{theorem}
\newtheorem{definition}[definition]{Definition}
\aliascntresetthe{definition}

\newaliascnt{exercise}{theorem}

\aliascntresetthe{exercise}

\theoremstyle{remark}
\newaliascnt{example}{theorem}

\aliascntresetthe{example}

\newaliascnt{remark}{theorem}
\newtheorem{remark}[remark]{Remark}
\aliascntresetthe{remark}

\usepackage{tikz-cd}
\usepackage{macros}
\usepackage{hyperref}
\usepackage{float}
\usepackage[capitalise,nameinlink]{cleveref}

\hypersetup{bookmarksopen=true}

\usepackage{crossreftools}
\crefformat{equation}{\textup{#2(#1)#3}}
\crefrangeformat{equation}{\textup{#3(#1)#4--#5(#2)#6}}
\crefmultiformat{equation}{\textup{#2(#1)#3}}{ and \textup{#2(#1)#3}}
{, \textup{#2(#1)#3}}{, and \textup{#2(#1)#3}}
\crefrangemultiformat{equation}{\textup{#3(#1)#4--#5(#2)#6}}%
{ and \textup{#3(#1)#4--#5(#2)#6}}{, \textup{#3(#1)#4--#5(#2)#6}}{, and \textup{#3(#1)#4--#5(#2)#6}}

\Crefformat{equation}{#2Equation~\textup{(#1)}#3}
\Crefrangeformat{equation}{Equations~\textup{#3(#1)#4--#5(#2)#6}}
\Crefmultiformat{equation}{Equations~\textup{#2(#1)#3}}{ and \textup{#2(#1)#3}}
{, \textup{#2(#1)#3}}{, and \textup{#2(#1)#3}}
\Crefrangemultiformat{equation}{Equations~\textup{#3(#1)#4--#5(#2)#6}}%
{ and \textup{#3(#1)#4--#5(#2)#6}}{, \textup{#3(#1)#4--#5(#2)#6}}{, and \textup{#3(#1)#4--#5(#2)#6}}

\crefname{theorem}{Theorem}{Theorems}
\crefname{proposition}{Proposition}{Propositions}
\crefname{corollary}{Corollary}{Corollaries}
\crefname{lemma}{Lemma}{Lemmas}
\crefname{definition}{Definition}{Definitions}
\crefname{remark}{Remark}{Remarks}

\usepackage[inline]{enumitem}
\usepackage{algorithm}
\title{Nonasymptotic Convergence Rates for Plug-and-Play Methods With MMSE Denoisers}

\author{Henry~Pritchard
        and~Rahul~Parhi,~\IEEEmembership{Member,~IEEE}\thanks{The authors acknowledge support from NVIDIA Corporation through the NVIDIA Academic Grant Program, including the donation of an RTX PRO 6000 Blackwell Max-Q Workstation Edition GPU used in this research. RP acknowledges support from a Hellman Fellowship from the University of California.

The authors are with the Department of Electrical and Computer Engineering, University of California, San Diego (e-mail: hepritchard@ucsd.edu; rahul@ucsd.edu).}}

\begin{document}
\maketitle

\bstctlcite{IEEEexample:BSTcontrol}

\begin{abstract}
    It is known that the minimum-mean-squared-error (MMSE) denoiser under Gaussian noise can be written as a proximal operator, which suffices for asymptotic convergence of plug-and-play (PnP) methods but does not reveal the structure of the induced regularizer or give convergence rates. We show that the MMSE denoiser corresponds to a regularizer that can be written explicitly as an upper Moreau envelope of the negative log-marginal density, which in turn implies that the regularizer is 1-weakly convex on its domain. Using this property, we derive (to the best of our knowledge) the first sublinear convergence guarantee for PnP proximal gradient descent with an MMSE denoiser. We validate the theory with a one-dimensional synthetic study that recovers the implicit regularizer. We further validate the theory with imaging experiments (deblurring and computed tomography), which exhibit the predicted sublinear behavior.
\end{abstract}

\section{Introduction}
In this work, we focus on the setting of \emph{linear inverse problems}, where $\overline{x} \in \mathbb{R}^n$ is the true signal, $y \in \mathbb{R}^m$ is the observed signal obtained by
\begin{equation}\label{eq:inverse_problem}
y = A\overline{x} + \epsilon, 
\end{equation}
for linear \emph{forward operator} $A \in \mathbb{R}^{m\times n}$ and measurement noise $\epsilon \sim \mathcal{N}(0,\sigma^2 I)$. The task of recovering $\overline{x}$ from $y$ is extensively studied in statistics and signal processing. A special case is the denoising problem, obtained when $A = I$: 
\begin{equation}\label{eq:denoising_problem}
    z = \overline{x}+\epsilon.
\end{equation}  
Although simpler, denoising is directly useful in solving general inverse problems, often as an intermediate step.

Recovering $\overline{x}$ from $y$ is often posed as the minimization of a regularized objective function
\begin{equation}\label{eq:loss_func}
 F(x) \coloneqq  f(x) + g(x),
\end{equation}
where $f$ is a data-fidelity term (typically $f(x) \coloneqq  \tfrac{1}{2}\|Ax-y\|^2$ under Gaussian noise), and $g$ is a regularization term that encodes prior beliefs about the \emph{class of admissible solutions}. For instance, the $\ell^1$-norm ($g(x) \propto \norm{x}_1$) promotes sparsity and underlies the theory of compressed sensing~\cite{candes_robust,donoho_compressed,bora_compressed}. 
Similarly, total variation (TV), defined as $g(x) \propto \TV(x) \coloneqq \sum_{i} |x_i - x_{i-1}|$ for $1$-D signals, favors piecewise-constant solutions and is widely used in image denoising~\cite{rudin_nonlinear,chambolle_algorithm,chambolle_first, liu_image,parhi2023sparsity}. Without regularization, the problem is often \emph{ill-posed}: Solutions may be non-unique, unstable, or nonexistent. Since the data-fidelity term is typically well-behaved, the regularizer largely determines the structure of solutions. More broadly, the regularizer largely dictates how optimization algorithms behave, so understanding its structure is crucial.

\subsection{Proximal Gradient Descent}
Gradient descent is a classical method for minimizing smooth (Lipschitz gradient) objective functions, with iterations
\begin{equation}
x_{k+1} = x_k - \gamma \nabla F(x_k).
\end{equation}
With appropriately chosen step size $\gamma > 0$, the iterates $x_k$ converge to a stationary point~\cite{kakade_dimension}. However, gradient descent is not well-suited for nonsmooth objectives.

Proximal gradient descent (PGD) addresses this limitation by splitting the objective function as in \eqref{eq:loss_func}. The data-fidelity term $f$ is typically smooth while the regularizer $g$ is often not (such as $\ell^1$ or $\TV$). At each iteration, PGD takes a gradient step on the data-fidelity term followed by a \emph{proximal step} for the regularizer, yielding iterations of the form
\begin{equation}\label{eq:pgd}
x_{k+1} = \operatorname{prox}_{\gamma g}
\left(x_k - \gamma \nabla f(x_k)\right).\end{equation}
Here $\operatorname{prox}_{\gamma g}$ denotes the proximal operator of $g$ (see \cref{def:prox}). Convergence of PGD has been widely studied in both convex and nonconvex settings and can be established under various assumptions on $f$ and $g$ with appropriately chosen step size $\gamma > 0$~\cite[Table~1]{li_accelerated}.

\subsection{Plug-and-Play Methods}
The \emph{plug-and-play} (PnP) framework~\cite{venkatakrishnan_plug} extends PGD by replacing the proximal step $\operatorname{prox}_{\gamma g}$ with a generic denoiser $D_{\sigma}$. This substitution is motivated by the interpretation of $\operatorname{prox}_{\gamma g}$ as the maximum a posteriori (MAP) solution to the denoising problem \eqref{eq:denoising_problem} under prior $p_X\propto \mathrm{e}^{-\gamma g}$~\cite[Eq.~(10.13)]{combettes_proximal}, the exact problem denoisers are designed to solve. The PnP-PGD iteration is given by
\begin{equation}\label{eq:pnp}
    x_{k+1} = D_\sigma
\left(x_k - \gamma \nabla f(x_k)\right),
\end{equation}
and has been shown to yield state-of-the-art results when combined with powerful image denoisers such as BM3D~\cite{dabov_image} or a deep neural network (DNN).

A central challenge in the PnP framework is that while a proximal operator is always tied to a well-defined regularizer, many of the most effective denoising methods are not. This substitution therefore breaks the direct connection to an explicit optimization problem, making theoretical interpretations and convergence claims difficult. Much of the literature imposes strict assumptions on the denoiser to obtain convergence guarantees. Many of the most effective denoisers do not satisfy these assumptions, yet still achieve excellent results~\cite{ahmad_plug}.

\begin{table*}[ht!]
 
\centering
\caption{Summary of related convergence guarantees for plug-and-play proximal gradient descent (PnP-PGD). Note $L_f$ is the Lipschitz constant for $\nabla f$, $L_{g_\sigma}$ is the Lipschitz constant for $\nabla g_\sigma$, $\gamma$ is the step size, and $K$ is the iteration count.
}
\label{tab:related_works}
\begin{tabular}{llll}
\toprule
\textbf{Work} & \textbf{Assumptions on data-fidelity term $f$} & \textbf{Assumptions on denoiser $D_\sigma$} & \textbf{Convergence rate} \\
\midrule
Classical PGD~\cite{beck_fast} & Convex, $L_f$-smooth & $D_\sigma = \operatorname{prox}_{g}$, where $g$ is convex & $\mathcal{O}(1/\sqrt{K})$\\
Lipschitz condition on the denoiser~\cite{ryu_plug} & $L_f$-smooth, $\mu$-strongly convex & $D_\sigma - I$ is $\varepsilon$-Lipschitz for $\varepsilon < \tfrac{2\mu}{L_f-\mu}$ & Asymptotic\\
Gradient-step denoisers~\cite{hurault2022gradient} & $L_f$-smooth, $\gamma L_f < \tfrac{L_{g_\sigma}+2}{L_{g_\sigma}+1}$ & $D_\sigma=I-\nabla g_\sigma$, where $L_{g_\sigma}<1$ & $\mathcal{O}(1/\sqrt{K})$ \\
Prior work with MMSE denoisers \cite{xu_provable} & $L_f$-smooth, $L_f<1$ & MMSE denoiser &  Asymptotic \\
\rowcolor{gray!20}
\textbf{This paper} (\cref{cor:conv}) & $L_f$-smooth, $L_f<1$ & MMSE denoiser & $\mathcal{O}(1/\sqrt{K})$ \\
\bottomrule
\end{tabular}
\end{table*}

\subsection{Related Works}
When $f$ and $g$ are convex with $L$-smooth $f$, it is well known that PGD converges to a fixed point for appropriate step size $\gamma$ \cite[Section 4.2]{parikh_prox}. This result can be proven via the machinery of monotone operator theory, the goal being to show that the composite PGD step \eqref{eq:pgd} constitutes an averaged operator, guaranteeing convergence~\cite{combettes_signal, combettes_solving, bauschke_convex}.\footnote{When $g$ is proper, lower-semicontinuous, convex, $\partial g$ is maximally monotone, making $\operatorname{prox}_{\gamma g} = (I+ \gamma\partial g)^{-1}$ firmly-nonexpansive, i.e., $1/2$-averaged. The gradient step is also averaged for $L$-smooth $f$, making their composition an averaged operator.}

In particular, convergence of PnP-PGD has been shown under tight assumptions on the denoiser, such as nonexpansiveness~\cite{sun_online} or suitable Lipschitz (e.g., residual-Lipschitz) bounds~\cite{ryu_plug}. Many works encourage these properties via, e.g., spectral normalization~\cite{ryu_plug} or Bj\"orck orthonormalization~\cite{anil2019sorting, li2019preventing}. In~\cite{pesquet_learning}, the authors propose a technique for learning a maximally monotone operator. Convergence of PnP has also been shown for specific denoisers, such as kernel denoisers~\cite{gavaskar_plug,athalye_contractivity}, gradient-step (GS) denoisers~\cite{hurault_proximal, hurault2022gradient}, bounded denoisers~\cite{chan_plug}, and linear denoisers~\cite{nair_fixed,gavaskar_plug_linear}.

A similar line of work learns regularizers explicitly~\cite{hertrich2025learning,ducotterd2025learning,ducotterd2025multivariate,goujon2023neural,pourya2025dealing,ongie2020deep, shumaylov_weakly, goujon_learning, parhi_deep}, rather than implicitly through the denoiser. This approach benefits from direct interpretability of the regularizer and can inherit convergence guarantees from classical variational theory. 

A complementary approach leverages the intrinsic \emph{statistical} structure of the denoiser. A particularly notable example---and the subject of this work---is the MMSE denoiser under Gaussian noise. With the Tweedie/Stein identities~\cite{efron_tweedies, stein_bound}, the MMSE denoiser can be tied to the score function, endowing it with useful structural properties~\cite{milanfar_denoising}. Further,~\cite{gribonval_reconciling} and~\cite{gribonval_should} showed that the MMSE denoiser is $C^\infty$ and can be written as the proximal map of a regularizer (which is $C^\infty$) under various forward and noise models, and derived an explicit formula for this regularizer~\eqref{eq:gribonval_explicit_mmse_prior}. This progress enabled the derivation of asymptotic convergence for both the ADMM~\cite{park_convergence} and PGD~\cite{xu_provable} variants of PnP using an MMSE denoiser. However, a gap in the theory remains: The explicit formula offers little insight into the behavior of the induced regularizer, and there are no \emph{nonasymptotic} convergence rates tailored to MMSE denoisers. \Cref{tab:related_works} highlights several closely related results on PnP-PGD, the assumptions they impose on the denoiser, and the convergence guarantees they establish.

\subsection{Contributions}
We make the following contributions.
\begin{itemize}
\item \textbf{Characterization of the implicit regularizer of the MMSE denoiser.}
We show in \cref{thm:mmse_prior} that the MMSE denoiser is the proximal map of a regularizer that admits an explicit form as an \emph{upper Moreau envelope}, i.e., $\phi_\mathrm{MMSE}(x)$ is, up to a constant,
\begin{equation}
\sigma^2\sup_{z}\left\{f_Z(z)-\tfrac{1}{2\sigma^2}\|z-x\|^2\right\},
\end{equation}
on its domain, where $f_Z=-\log p_Z$ is the negative log-marginal density from \cref{eq:denoising_problem}. This is a novel explicit form of the implicit regularizer in the MMSE denoiser.
\item \textbf{Nonasymptotic convergence of MMSE-PGD via weak convexity.}
In \cref{cor:wcvx}, we prove that $\phi_{\mathrm{MMSE}}$ is $1$-weakly convex on its domain. With this property, we establish (to the best of our knowledge) the first \emph{nonasymptotic} convergence guarantees for PnP-PGD with an MMSE denoiser in \cref{thm:PGD_MMSE_convergence} and \cref{cor:conv}. In particular, we show that, under an $L$-smooth data-fidelity term $f$ with $L<1$, both $\|\nabla F(x_k)\|$ and $\|x_{k+1}-x_k\|$ decrease sublinearly in terms of the best iterate.
\item \textbf{Experiments.} In \cref{Experiments}, we consider the following experimental setups: (i) a one-dimensional synthetic example validating the upper Moreau envelope form of the implicit regularizer for the MMSE denoiser, (ii) Gaussian deblurring on the MNIST dataset~\cite{lecun_mnist}, and (iii) computed tomography on the MayoCT dataset~\cite{mcollough_low}. In both imaging tasks we observe the predicted sublinear convergence.
\end{itemize}

\section{Preliminaries}

In this section, we collect some definitions and results used in the remainder of the paper. 

\begin{definition}[Lower Moreau Envelope]\label{def:lower_moreau}
   Let $f:\mathbb{R}^n\to \mathbb{R}$ be a function with $\gamma>0$. The \emph{lower Moreau envelope} of $f$ is given by 
    \begin{equation}
    M_{\gamma }f(x) \coloneqq  \inf_y \left\{f(y)+\tfrac{1}{2 \gamma}\|y-x\|^2 \right\}.\end{equation}
\end{definition}

\begin{definition}[Upper Moreau Envelope]\label{def:upper_moreau}
    Let $f:\mathbb{R}^n \to \mathbb{R}$ be a function with $\gamma > 0$. The \emph{upper Moreau envelope} of $f$ is given by
    \begin{equation}M^{\gamma}f(x)\coloneqq \sup_y \left\{f(y)-\tfrac{1}{2 \gamma}\|y-x\|^2\right\}.\end{equation}
\end{definition}
Note that the lower Moreau envelope is the usual Moreau envelope, but we make the distinction in this paper since we use both lower and upper Moreau envelopes. For the infimum defining the lower Moreau envelope to be achieved, $f$ must be lower-semicontinuous. An important property of the lower Moreau envelope $M_{\gamma}f$ is that it shares global minimizers with $f$ whenever $f$ is lower-semicontinuous, i.e. $\argmin_x f(x) = \argmin_{x} M_{\gamma}f(x).$
See~\cite{bernard_lasry} for further background.

\begin{definition}[Weak convexity]\label{def:weak}
Let $f:\mathbb{R}^n\to \R$, $\gamma>0$. $f$ is called \emph{$\gamma$-weakly convex} if the mapping $x \mapsto f(x)+\tfrac{\gamma}{2}\|x\|^2$ is convex. 
\end{definition}

Crucial to our analysis is the following lemma, adapted from~\cite[Lemma~3]{bernard_lasry}.
\begin{lemma}\label{lem:bernard3}
Let $f:\mathbb{R}^n \to (-\infty,\infty]$. Then, $M^\gamma M_\gamma f(x) \leq f(x)$ for all $x\in \mathbb{R}^n$ and $M^\gamma f$ is $\tfrac{1}{\gamma}$-weakly convex.
\end{lemma}
\begin{proof}
By definition,
\begin{align}
&M^\gamma M_\gamma f(x)  \nonumber \\
={}&\sup_y \inf_z \left\{f(z) 
     + \tfrac{1}{2\gamma}\|z-y\|^2 \nonumber
     - \tfrac{1}{2\gamma}\|y-x\|^2 \right\} \\
\leq{}& \sup_y f(x) = f(x).
\end{align}
Next, 
\begin{align}
    &M^\gamma f(x) + \tfrac{1}{2\gamma}\|x\|^2 \nonumber \\
    ={}& \sup_y\left\{ f(y) -\tfrac{1}{2\gamma}\|x-y\|^2 + \tfrac{1}{2\gamma}\|x\|^2\right\} \nonumber \\
    ={}& \sup_y\left\{ f(y)-\tfrac{1}{2\gamma}\|y\|^2 + \tfrac{1}{\gamma}\langle x,y\rangle\right\},
\end{align}
which is a supremum of affine functions, and hence convex. Therefore, $M^\gamma f$ is $\tfrac{1}{\gamma}$-weakly convex.
\end{proof}
 
We now turn to the proximal operator, a central tool in optimization. We shall rely on the fact that MMSE denoisers themselves admit a proximal representation.

\begin{definition}[Proximal operator]\label{def:prox}
Let $\gamma>0$, and $f:\mathbb{R}^n\to\mathbb{R}\cup \{\infty\}$ be proper and lower-semicontinuous.
The \emph{proximal operator} of $f$ with parameter $\gamma$ is the mapping
\begin{equation}\label{eq:prox}
\operatorname{prox}_{\gamma f}(x)\;\coloneqq \;\operatorname*{arg\,min}_{z\in \mathbb{R}^n}
\left\{f(z)+\tfrac{1}{2\gamma}\|z-x\|^2\right\}.
\end{equation}
\end{definition}
The proximal operator is generally set-valued and may be empty unless additional assumptions are imposed. 
For a proper, lower-semicontinuous function $f:\mathbb{R}^n\to \mathbb{R}\cup\{\infty\}$, the infimum defining the lower Moreau envelope is attained at every minimizer of the proximal problem; that is, 
\begin{equation}\label{eq:moreau_and_prox}
M_{\gamma} f(x) = f\left(z\right)+\tfrac{1}{2\gamma}\|x-z\|^2, \quad \text{for all } z \in \operatorname{prox}_{\gamma f}(x),
\end{equation}
for all $x \in \mathbb{R}^n$. When $f$ is a convex, lower-semicontinuous, and proper function, it is well known that $\operatorname{prox}_{\gamma f}$ and $M_{\gamma} f$ are also related by the following formula due to Moreau~\cite{moreau_proximite}: 
\begin{equation} 
\nabla M_{\gamma} f(x) = \tfrac{1}{\gamma}(x-\operatorname{prox}_{\gamma f}(x)).
\end{equation}

Although the proximal operator of a nonconvex function is generally set-valued, when $\operatorname{prox}_{\gamma f}$ is an MMSE denoiser, it is always single-valued and continuous, which is the setting of our main results. This motivates the following extension of the Moreau gradient identity. 

\begin{theorem}[Extension of the Moreau Gradient Identity]\label{thm:moreau_gradient}
  Let $f:\mathbb{R}^n\to \mathbb{R}\cup\{\infty\}$ be proper and lower-semicontinuous. If $\operatorname{prox}_{\gamma f}(x)$ is single-valued and continuous for all $x$, then the gradient of the lower Moreau envelope exists and takes the form
  \begin{equation}
      \nabla M_{\gamma} f(x) = \tfrac{1}{\gamma}(x-\operatorname{prox}_{\gamma f}(x)), \quad x \in \R^n.
  \end{equation}
\end{theorem}

\begin{proof}
    Let $h(x,z)\coloneqq  f(z)+\tfrac{1}{2\gamma}\|z-x\|^2$. We have for all $z \in \mathbb{R}^n$, $t\in \mathbb{R}$, 
    \begin{equation}\label{eq:2_7_1}
    h(x+tu,z) -h(x,z) = \tfrac{t}{\gamma}\langle x-z,u\rangle +\tfrac{t^2}{2\gamma}\|u\|^2.
\end{equation}
From the definition, $M_{\gamma} f(x) = \min_{z} h(x,z)$. Fix some $x_0 \in \mathbb{R}^n$  where there is a unique minimizer $z_0 = \operatorname{prox}_{\gamma f}(x_0)$ by assumption. Take some $u \in \mathbb{R}^n$. For $t>0$, there is some $z_t = \operatorname{prox}_{\gamma f}(x_0+ tu)$ so that
\begin{align}
&M_{\gamma} f(x_0+tu) \nonumber \\
={}& h(x_0+tu,z_t) \nonumber \\ 
={}& h(x_0,z_t) + \tfrac{t}{\gamma}\langle x_0-z_t,u\rangle+\tfrac{t^2}{2\gamma}\|u\|^2  \nonumber \\ 
\geq{}& M_{\gamma} f(x_0) +  \tfrac{t}{\gamma}\langle x_0-z_t,u\rangle+\tfrac{t^2}{2\gamma}\|u\|^2.
\end{align}
Subtracting and dividing through by $t$, we get \begin{align}\label{eq:2_7_2}
    & \liminf_{t\downarrow 0} \tfrac{M_{\gamma} f(x_0+tu)-M_{\gamma} f(x_0)}{t} \nonumber \\ \geq{}& \liminf_{t\downarrow 0} \left( \tfrac{1}{\gamma}\langle x_0-z_t,u\rangle+\tfrac{t}{2\gamma}\|u\|^2 \right)\nonumber  \\  ={}&\tfrac{1}{\gamma}\langle x_0-z_0,u\rangle,
\end{align}
where the last equality is by the assumption that $\operatorname{prox}_{\gamma f}$ is continuous. Next, 
\begin{align} & M_{\gamma} f(x_0+tu) \nonumber \leq h(x_0+tu, z_0) \\={}& M_{\gamma} f(x_0) + \tfrac{t}{\gamma}\langle x_0-z_0,u\rangle + \tfrac{t^2}{2\gamma}\|u\|^2.\end{align}
Again, subtracting and dividing through by $t$,
\begin{align}\label{eq:2_7_3}
    & \limsup_{t\downarrow 0} \tfrac{M_{\gamma} f(x_0+tu) -  M_{\gamma} f(x_0)}{t} \nonumber \\\leq{}& \limsup_{t \downarrow 0}\left(\tfrac{1}{\gamma} \langle x_0-z_0,u\rangle + \tfrac{t}{2\gamma}\|u\|^2 \right)\nonumber  \\ ={}& \tfrac{1}{\gamma} \langle x_0-z_0,u\rangle.
\end{align}
Combining \eqref{eq:2_7_2} and \eqref{eq:2_7_3}, we have that the directional derivative of $M_{\gamma} f$ at $x_0$ in direction $u$ satisfies $(M_{\gamma} f)'(x_0,u) = \tfrac{1}{\gamma}\langle x_0-z_0,u\rangle$. Linearity in $u$ gives us a two-sided directional derivative. Therefore, the gradient is
\begin{equation}
\nabla M_{\gamma} f(x_0) = \tfrac{1}{\gamma}(x_0-z_0) = \tfrac{1}{\gamma} \left(x_0-\operatorname{prox}_{\gamma f}(x_0)\right),
\end{equation}
which completes the proof.
\end{proof}
Many similar versions of this identity are well known in the literature, but are often formulated in terms of \textit{prox-regularity}. We included a simple self-contained proof for completeness. For more detail, see \cite[Theorem~13.37]{rockafellar_variational}.

We next review key results on MMSE denoisers and state their connection to proximal mappings, beginning with Tweedie’s formula.
\begin{theorem}[Tweedie's Formula]\label{thm:tweedies}
Let $X \sim p_X$, $\epsilon \sim \mathcal{N}(0,\sigma^2I )$ be independent random variables in $\R^n$, and let $Z=X+\epsilon$, $\psi_{\sigma}(z)\coloneqq \mathbb{E}[X|Z=z]$. Then, \begin{equation}\label{eq:tweedies}
    \psi_{\sigma}(z) = z+ \sigma^2 \nabla \log p_Z(z),
\end{equation}
where $p_Z= p_X * \mathcal{N}(0,\sigma^2I )$.
\end{theorem}
For a thorough introduction to Tweedie’s formula, we refer the reader to~\cite{efron_tweedies}. Note that Tweedie's formula makes no assumptions on the prior distribution $p_X$, which may even include discrete atoms. We recall the following result on MMSE denoisers due to Gribonval et al.~\cite{gribonval_reconciling,gribonval_should}.

\begin{theorem}[Main result from \cite{gribonval_should}~and~{\cite[Corollary~1]{gribonval_reconciling}}]\label{thm:gribonval_mmse_result}
Let $X \sim p_X$ and $\epsilon \sim \mathcal{N}(0,\sigma^2I )$ be independent random variables in $\R^n$, and let $Z=X+\epsilon$. Assume $p_X$ is non-degenerate.\footnote{i.e., there is no pair $v\in \R^n$, $c\in \R$ such that $\langle X,v\rangle=c$, almost surely.} The MMSE denoiser under \eqref{eq:denoising_problem} given by $\psi_\sigma(z) \coloneqq  \mathbb{E}[X|Z=z]$ has the following properties:\begin{enumerate} [label=(\roman*)]
    \item It is one-to-one from $\R^n$ onto its image.
    \item It is $C^\infty$, and so is its inverse.
    \item It can be written as $\psi_\sigma(z)=\operatorname{prox}_{\phi_{\mathrm{MMSE}}}(z)$, where
    \end{enumerate}
    \begin{align}
     &\phi_{\mathrm{MMSE}}(z)\coloneqq  \nonumber \\
    &\begin{cases}
    -\tfrac12\|\psi_\sigma^{-1}(z)-z\|^2 +\sigma^2f_Z\left(\psi_\sigma^{-1}(z)\right), & z \in \psi_\sigma(\mathbb{R}^n),\\
    \infty, & \text{otherwise},
    \end{cases}\label{eq:gribonval_explicit_mmse_prior}
    \end{align}
    \begin{enumerate}[label=(\roman*),resume]
    \item[]where $f_Z = -\log p_Z$.
    \item $\phi_{\mathrm{MMSE}}$ is $C^\infty$ on $\psi_{\sigma}(\mathbb{R}^n)$.
\end{enumerate}
\end{theorem}

\section{Main Results}\label{sec:main}
In this section, we present our theoretical contributions. We begin by introducing a new structural representation of the implicit regularizer underlying the MMSE denoiser. This characterization serves as the foundation for deriving novel nonasymptotic convergence guarantees.

\subsection{New Characterization of the MMSE Regularizer}\label[subsection]{section:characterization}
It is known (\cref{thm:gribonval_mmse_result}) that the MMSE denoiser is the proximal map of an infinitely differentiable penalty. Yet beyond this, the structure of the associated penalty \(\phi_{\mathrm{MMSE}}\) remains unclear. The only known explicit expression of $\phi_{\mathrm{MMSE}}$ \eqref{eq:gribonval_explicit_mmse_prior} is cumbersome and offers little analytical insight. In this section, we show that \(\phi_{\mathrm{MMSE}}\) admits an upper Moreau envelope representation in terms of $f_Z$, yielding useful insights, namely $1$-weak convexity, which we will leverage to prove stronger convergence results for PGD using an MMSE denoiser.

For the following, we will assume the setting of the denoising problem \eqref{eq:denoising_problem}, letting $X\sim p_X$, $\epsilon \sim \mathcal{N}(0,\sigma^2I )$ be independent random variables in $\R^n$, $Z=X+\epsilon$, so that $p_Z = p_X*\mathcal{N}(0,\sigma^2I )$. We will let $f_Z\coloneqq -\log p_Z$.

\begin{lemma}\label{lem:moreau_identity}
The negative log marginal distribution from the denoising problem \eqref{eq:denoising_problem} can be written as
\begin{equation}
f_Z = \tfrac{1}{\sigma^2}M_1 \phi_{\mathrm{MMSE}}+ C_{x_0},
\end{equation}
where $\phi_{\mathrm{MMSE}}$ is as in \eqref{eq:gribonval_explicit_mmse_prior} and $C_{x_0}$ takes the form 
\begin{equation}\label{eq:constant}
C_{x_0} \coloneqq  f_Z(x_0) - \tfrac{1}{\sigma^2}M_1 \phi_{\mathrm{MMSE}}(x_0),
\end{equation}
for any choice of $x_0 \in \mathbb{R}^n$.
        
\begin{proof} Pick $x\in \mathbb{R}^n$. By \cref{thm:tweedies} and \cref{thm:gribonval_mmse_result}, 
\begin{align}
x-\operatorname{prox}_{\phi_{\mathrm{MMSE}}}(x)& =-\sigma^2\nabla \log p_Z(x) \nonumber
\\
& = \sigma^2 \nabla f_Z(x).
\end{align}
Combined with \cref{thm:moreau_gradient} yields 
\begin{equation}
    \sigma^2 \nabla f_Z(x) = \nabla M_1 \phi_{\mathrm{MMSE}}(x).
\end{equation}
Integrating gives the result.
\end{proof}
\end{lemma}
 
\begin{theorem}\label{thm:mmse_prior}
The implicit regularizer in the MMSE denoiser can be written as 
\begin{equation}\label{eq:mmse_prior_explicit}
\phi_{\mathrm{MMSE}}(x) = \begin{cases} \sigma^2M^{\sigma^2}f_Z (x) - \sigma^2C_{x_0}, & x\in \psi_{\sigma}(\mathbb{R}^n)\\
+\infty, & \text{otherwise},
\end{cases}
\end{equation} where $C_{x_0}$ is defined in \eqref{eq:constant}.
\end{theorem}
 
\begin{proof}
Pick some $x \in \mathbb{R}^n$. By \cref{lem:bernard3},
\begin{equation}
     \phi_{\mathrm{MMSE}}(x) \geq M^{1} M_{1}\phi_{\mathrm{MMSE}}(x),
\end{equation}
which, by the previous result, we can rewrite as 
\begin{align}
\phi_{\mathrm{MMSE}}(x) &\geq  M^1(\sigma^2 f_Z-\sigma^2 C_{x_0})(x) \nonumber \\ &= \sigma^2 M^{\sigma^2}f_Z(x)-\sigma^2 C_{x_0}.
\end{align}
For the reverse inequality, take some $x\in \psi_{\sigma}(\mathbb{R}^n)$. There exists some $z\in \mathbb{R}^n$ so that $x=\psi_{\sigma}(z) = \operatorname{prox}_{ \phi_{\mathrm{MMSE}}}(z)$. The infimum defining $M_1 \phi_{\mathrm{MMSE}}$ is achieved at $x$ (cf.~\eqref{eq:moreau_and_prox}), so that
\begin{align}
\phi_{\mathrm{MMSE}}(x) + \tfrac{1}{2}\|x-z\|^2 &= M_{1} \phi_{\mathrm{MMSE}}(z) \nonumber \\
&= \sigma^2f_Z(z) - \sigma^2C_{x_0},
\end{align}
where the second line holds by \cref{lem:moreau_identity}.
Rearranging, we have 
\begin{align}
 \phi_{\mathrm{MMSE}}(x) &= \sigma^2 \left(f_Z(z) - \tfrac{1}{2\sigma^2}\|x-z\|^2 
\right)- \sigma^2C_{x_0} \nonumber \\
&\leq \sigma^2M^{\sigma^2}f_Z(x) - \sigma^2 C_{x_0}, 
\end{align}
which completes the proof.
\end{proof}
We illustrate this result in \cref{fig:1.1} and \cref{fig:1.2} under a variety of mixture-of-Gaussian and mixture-of-Laplacian priors. An overview of the experimental details is provided in \cref{section:illustration}. 
We conclude with an immediate corollary.

\begin{corollary} \label{cor:wcvx}
The implicit regularizer $\phi_{\mathrm{MMSE}}$ is $1$-weakly convex on its domain $\mathrm{dom}(\phi_{\mathrm{MMSE}}) =  \psi_{\sigma}(\mathbb{R}^n)$.
\end{corollary} 
\begin{proof}
By~\cref{thm:mmse_prior}, $\phi_{\mathrm{MMSE}}$ equals $\sigma^2M^{\sigma^2}f_Z$ on $\psi_\sigma(\mathbb{R}^n)$. By \cref{lem:bernard3}, $M^{\sigma^2}f_Z$ is $\tfrac{1}{\sigma^2}$-weakly convex. Hence, $\sigma^2M^{\sigma^2}f_Z$ is $1$-weakly convex. The result follows.
\end{proof}

\begin{figure*}[htb!]
    \centering
    \includegraphics[width=\linewidth]{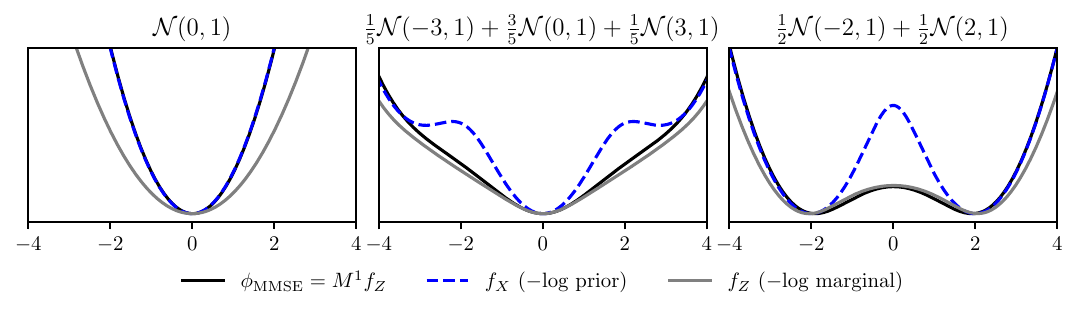}
    \centering
    \includegraphics[width=\linewidth]{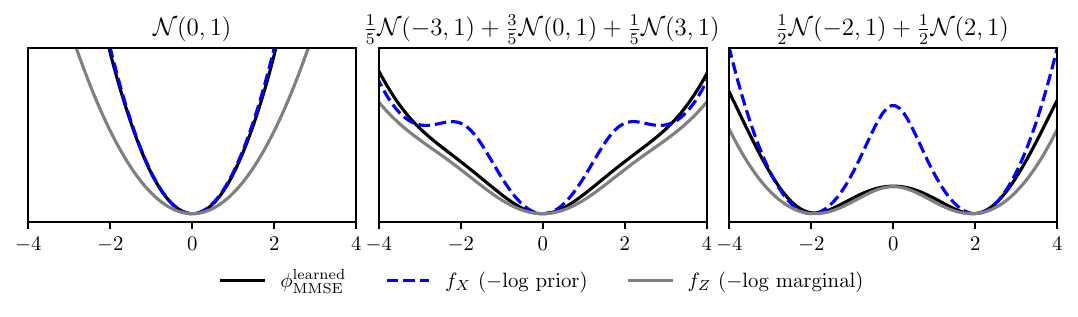}
    \caption{Top: Calculated implicit regularizer in the MMSE denoiser ($\phi_{\mathrm{MMSE}}$), negative log-marginal $f_Z$, and negative log-prior $f_X$ for three mixture-of-Gaussian priors under unit Gaussian noise. Bottom: The learned regularizer $\phi^{\mathrm{learned}}_{\mathrm{MMSE}}$, along with the same reference curves. Details in \cref{section:illustration}.}
    \label{fig:1.1}
\end{figure*}
\begin{figure*}[htb!]
    \centering
    \includegraphics[width=\linewidth]{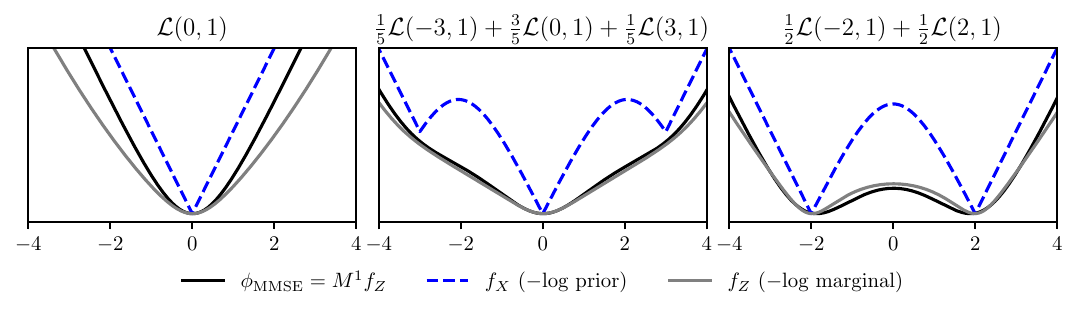}
    \centering
    \includegraphics[width=\linewidth]{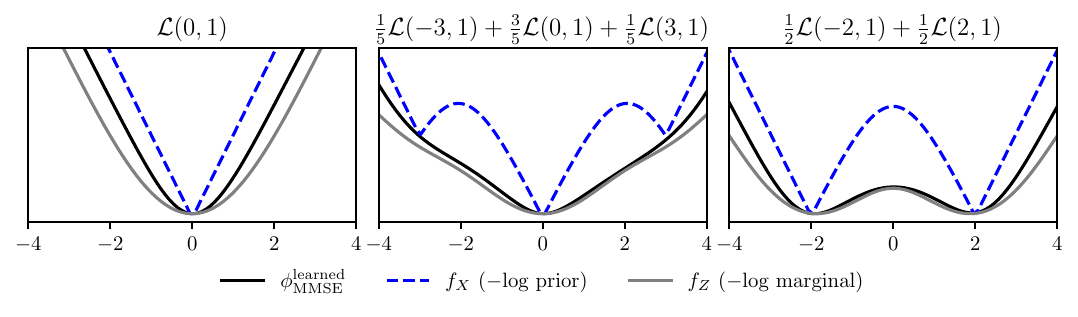}
    \caption{Top: Calculated implicit regularizer in the MMSE denoiser ($\phi_{\mathrm{MMSE}}$), negative log-marginal $f_Z$, and negative log-prior $f_X$ for three mixture-of-Laplacian priors under unit Gaussian noise. Bottom: The learned regularizer $\phi^{\mathrm{learned}}_{\mathrm{MMSE}}$, along with the same reference curves. Details in \cref{section:illustration}.} 
    \label{fig:1.2}
\end{figure*}

\subsection{Proximal Gradient Descent Convergence Results}

Consider the iteration
\begin{align}\label{eq:MMSE_PGD}
    x_{k+1} = \psi_\sigma(x_k -\nabla f(x_k)),
\end{align}
where $\psi_\sigma$ is the MMSE denoiser under Gaussian noise of variance $\sigma^2$ and $f$ is an $L$-smooth data-fidelity term. The step size is fixed because it is implicit in the MMSE denoiser. The associated objective for these iterates is
\begin{align}\label{eq:MMSE_PGD_objective}
    F(x)\coloneqq f(x) + \phi_{\mathrm{MMSE}}(x),
\end{align}
where $\phi_{\mathrm{MMSE}}$ is the implicit regularizer in the MMSE estimator, discussed in \cref{section:characterization}. 

When the objective $F$ is smooth, the map $x\mapsto \nabla F(x)$ is continuous. Therefore, measuring convergence via the gradient norm $\|\nabla F(x_k)\|$ is natural as small gradient norms imply closeness to a stationary point. For nonconvex problems, $\|\nabla F(x_k)\|$ may oscillate even as it trends downward, so nonasymptotic convergence is typically stated in terms of the best iterate up until the current iteration~\cite{attouch_convergence}. Since $f$ is smooth and $\phi_{\mathrm{MMSE}}$ is $C^\infty$ on its domain (\cref{thm:gribonval_mmse_result}), the objective $F$ is $C^\infty$ on its domain. We measure convergence via
$\min_{1\leq i \leq K} \|\nabla F(x_i)\|$ at each iteration $K$. We now state our main convergence result. As we shall see, once weak convexity is established, sublinear convergence follows immediately by specializing existing theory to our context.
\begin{theorem}[Convergence of \cref{eq:MMSE_PGD}]\label{thm:PGD_MMSE_convergence}
Let $f$ be proper, bounded from below and $L$-smooth with $L<1$. Then for $F$ in~\cref{eq:MMSE_PGD_objective} and $x_k$ as in~\eqref{eq:MMSE_PGD}, the following hold:
\begin{enumerate} [label=(\roman*)]
    \item The sequence $(F(x_k))$ is non-increasing and converges.
    \item \label{two} The sequence has finite squared length, i.e.,
    $\sum_{k=0}^{\infty} \|x_{k+1} - x_k\|^2 < \infty$,
   and the best-step residual satisfies
    \begin{equation}
        \min_{1\leq i \leq K} \|x_{i+1} - x_i\| = \mathcal{O}\left(\tfrac{1}{\sqrt{K}}\right).
    \end{equation}
    \item All cluster points of $x_k$ are stationary points of $F$.
\end{enumerate}
\end{theorem}

\begin{proof}
    We have that $\sigma^2 M^{\sigma^2} f_Z-\sigma^2C_{x_0}$ is a globally-defined $1$-weakly convex function that agrees with $\phi_{\mathrm{MMSE}}$ on $\psi_\sigma(\mathbb{R}^n)$ by~\cref{thm:mmse_prior} and~\cref{lem:bernard3}.  Further, 
    \begin{align}
        \sigma^2M^{\sigma^2}f_Z - \sigma^2C_{x_0} 
        &\geq \sigma^2f_Z - \sigma^2C_{x_0}  \nonumber \\
        &\geq \sigma^2\tfrac{n}{2}\log (2\pi\sigma^2) - \sigma^2C_{x_0}, 
        \end{align}
        so $\phi_{\mathrm{MMSE}}$ is bounded below. 
    Next, each iterate of~\eqref{eq:MMSE_PGD} clearly satisfies $x_k \in \psi_\sigma(\mathbb{R}^n)$ for $k \geq 1$.
    Define \begin{align}
    \widetilde{\phi} =\sigma^2 M^{\sigma^2}f_Z-\sigma^2C_{x_0}.    
    \end{align}
    By~\cref{thm:mmse_prior} and \cref{lem:bernard3}, $\widetilde{\phi}$ is globally $1$-weakly convex and agrees with $\phi_{\mathrm{MMSE}}$ on $\psi_\sigma(\mathbb{R}^n)$ and hence on the iterates $x_k$ for $k \geq 1$. Applying the descent argument from~\cite[Theorem 1]{hurault_convergent} with $\tau=1$, $M=1$, and $L_f < 1$ yields the result.
    
\end{proof}
\begin{corollary}\label{cor:conv}
    Under the assumptions of \cref{thm:PGD_MMSE_convergence}, 
    \begin{align}
        \min_{1 \leq i \leq K} \|\nabla F(x_i)\| = \mathcal{O} \left( \tfrac{1}{\sqrt{K}} \right).
    \end{align}
\end{corollary}
\begin{proof}
From~\cref{thm:gribonval_mmse_result} and~\cref{thm:mmse_prior}, 
    \begin{align}x_{k+1} = \argmin_{z} G_k(z),\end{align}
    where
    \begin{align}G_k(z) & \coloneqq \sigma^2(M^{\sigma^2}f_Z(z) -C_{x_0})\nonumber \\  &\quad + \tfrac{1}{2}\|z-(x_k - \nabla f(x_k))\|^2.\end{align} By~\cref{lem:bernard3}, $\sigma^2M^{\sigma^2}f_Z$ is $1$-weakly convex. Hence, $G_k$ is convex and we can conclude that its minimizer satisfies first-order optimality, i.e.,
    $\nabla G_k(x_{k+1}) = 0$. We can write this as 
        $\nabla \phi_{\mathrm{MMSE}}(x_{k+1}) + x_{k+1} - (x_k - \nabla f(x_k)) = 0,$
    because $\phi_{\mathrm{MMSE}}(z)$ and $\sigma^2M^{\sigma^2}f_Z(z) -\sigma^2C_{x_0}$ agree on the open set $\psi_{\sigma}(\mathbb{R}^n)$, and $x_{k+1} \in \psi_{\sigma}(\mathbb{R}^n)$.  Therefore,
    \begin{align}\label{eq:gradient_calculation}
        \nabla F(x_{k+1}) = x_k - x_{k+1} + \nabla f(x_{k+1}) - \nabla f(x_k).
    \end{align}
    From this, we have that
    \begin{align}
        \|\nabla F(x_{k+1})\| & \leq \|x_k - x_{k+1}\| + \|\nabla f(x_{k+1})-\nabla f(x_k)\| \nonumber\\
        &\leq (1+L)\|x_{k+1}-x_k\| < 2 \|x_{k+1}-x_k\|. 
        \end{align}
        The result then follows from \cref{thm:PGD_MMSE_convergence}~\ref{two}. 
\end{proof}

\begin{remark}
The requirement $L<1$ may seem restrictive. However, any $L$-smooth data-fidelity term can be scaled appropriately to meet this condition. For example, in our computed tomography experiment (\cref{section:CT}), the forward operator is a discretized Radon transform, for which $f$ is $L$-smooth, $L \approx 140$. We simply scale the data-fidelity term to obtain empirical results that verify our theoretical convergence rate.
\end{remark}
 
\begin{remark}
The \emph{asymptotic} convergence of PGD using an MMSE denoiser has previously been established for a range of step sizes $\gamma>0$. In particular,~\cite{xu_provable} showed that PnP-MMSE converges for any $\gamma \leq 1/L$. Though the MMSE denoiser can be written as $\operatorname{prox}_{\phi_{\mathrm{MMSE}}}$, we note that $\operatorname{prox}_{\gamma \phi_{\mathrm{MMSE}}}$ cannot be evaluated for any $\gamma \neq 1$. Consequently, any change in the step size can be realized by a rescaling of the data-fidelity term. For this reason, we assume an implicit step size of $1$ and scale the Lipschitz constant of the data-fidelity term accordingly. 
\end{remark}

\section{Experiments} \label{Experiments}

In this section, we numerically verify the theoretical results presented in \cref{sec:main}. The code to reproduce our experiments is publicly available on GitHub.\footnote{\url{https://github.com/sparsity-group/pnpmmse}} All training and testing were performed on an NVIDIA RTX PRO 6000 Blackwell \mbox{Max-Q} Workstation Edition GPU using Ubuntu 24.04.

Central to our experiments is the task of estimating the MMSE denoiser of a given distribution under Gaussian noise. Following the work of~\cite{fang_whats}, we choose to model the MMSE denoiser as the gradient of a softplus deep neural network (DNN). We found that this architecture performed better than rectified linear unit (ReLU) alternatives. We also incorporate skip connections at each layer, as
they have been shown to aid in image denoising tasks~\cite{mao_image}. Finally, by Tweedie's formula (\cref{thm:tweedies}), the \textit{true} MMSE denoiser can be written as the gradient of $\tfrac{1}{2}\|\cdot\|^2 + \sigma^2 \log p_Z$, where $p_Z$ is the marginal distribution under Gaussian noise. For this reason, our choice of modeling the MMSE denoiser as the gradient of a DNN is well-motivated.

\begin{table*}[!t]
 
\centering
\caption{DNN architectural details for inverse problems experiments (\cref{section:deblurring} and \cref{section:CT}). $H$ denotes the hidden dimension and $L$ denotes the number of layers.}
\label{tab:architecture}
\begin{tabular}{llll}
\toprule
& \textbf{} & \textbf{Gaussian Deblurring} & \textbf{Computed Tomography} \\
\midrule
\textbf{Input size} & & $28\times28$ & $512\times512$ \\
\textbf{Dimensions} & & $H \times 28\times28,\ H\times14\times14,\ H \times 14\times14,$ & $H\times 512\times512,\ H \times 256\times256,\ H \times 256\times256,\ H \times 128\times128$ \\
& & $H \times 7\times7,\ H \times 7^2 \times 1 \times 1,\ 64 \times 1 \times 1$ & $H \times 128\times128,\ H \times 64\times64,\ H \times 64\times64,\ H \times 32\times32,$ \\
& & & $H \times 16\times16,\ 64 \times 1 \times 1,\ 64 \times 1 \times 1$ \\
$L$ & & $6$ & $11$ \\
$H$ & & $64$ & $256$ \\
\textbf{Direct connections} ($W_k$) & & 4 Conv2d ($3\times3$ kernel), 2 Linear & 9 Conv2d ($3\times3$ kernel), 1 Conv2d ($16\times16$ kernel), 1 Linear \\
\textbf{Skip connections} ($A_k$) & & 3 Conv2d, 1 Linear & 8 Conv2d ($3\times3$ kernel), 1 Conv2d ($16\times16$ kernel) \\
\bottomrule
\end{tabular}
\end{table*}

\begin{figure}[!htbp]
    \centering
    \includegraphics[width=0.9\linewidth]{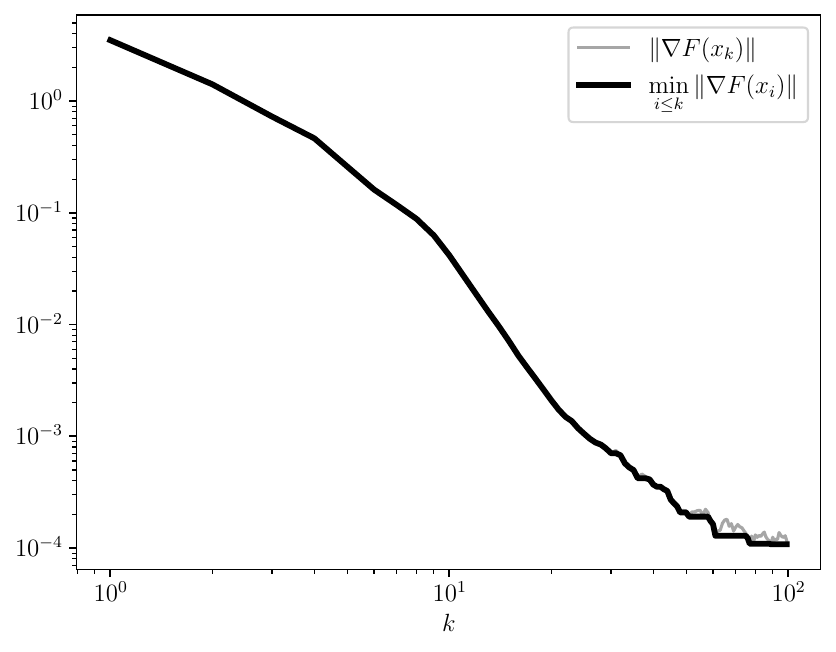}
    \caption{Best-iterate gradient norm over $100$ iterations of PnP-PGD with an MMSE denoiser for the MNIST Gaussian deblurring experiment in \cref{section:deblurring}, plotted on a log-log scale. The quantity decays sublinearly, consistent with a $\mathcal{O}(1/\sqrt{K})$ rate.}
    \label{fig:1.4}
\end{figure}

\subsection{Numerical Illustration of \cref{thm:mmse_prior}}\label[subsection]{section:illustration}

\begin{figure}[t] 
    \centering
    \includegraphics[width=0.9\linewidth]{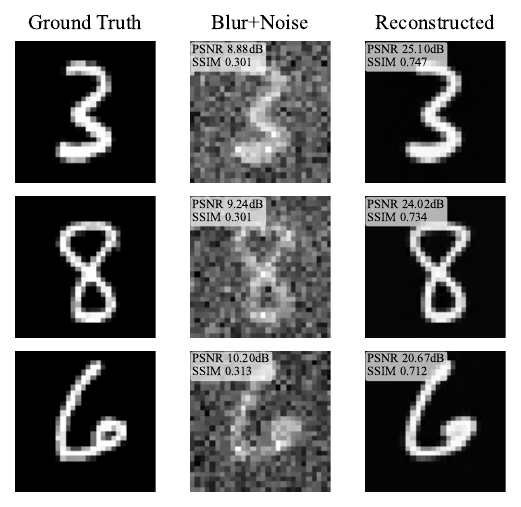}
   \caption{PnP-PGD reconstruction results on three selections from the MNIST dataset under Gaussian blur (\cref{section:deblurring}). From left to right: clean reference, blurred and noisy input, and reconstruction after 100 iterations. PSNR gains are roughly $10$--$16$~dB across all three examples.}
    \label{fig:1.3}
\end{figure}

We first consider the denoising problem \cref{eq:denoising_problem} with simple 1-dimensional mixture-of-Gaussian and mixture-of-Laplacian priors, denoted by $p_X$, with additive zero-mean, unit-variance Gaussian noise $\epsilon$. In each case, $p_Z$ denotes the marginal distribution of $z$. We now describe the DNN architecture used to learn the MMSE denoiser. Let
\begin{align}
    h_1 &\coloneqq \sigma(W_1x), \quad W_1 \in \mathbb{R}^H, \nonumber \\
    h_{k} &\coloneqq \sigma(W_{k}h_{k-1} + a_{k-1} x + b_{k-1}),\quad k = 2, \ldots, L, 
\end{align}
and $g_\theta(x) \coloneqq w_{\mathrm{out}}^\T h_L + a_{\mathrm{out}}x + b_\mathrm{out}$, where $W_k \in \mathbb{R}^{H \times H}$ are the weights, $a_k, b_k \in \mathbb{R}^H$ are the skip connections and biases, and $w_{\mathrm{out}},a_{\mathrm{out}},b_{\mathrm{out}} \in \mathbb{R}$ are the output weights and biases. 

We use a softplus activation function
\begin{align}\label{eq:softplus}
    \sigma(x)\coloneqq \frac{1}{\beta}\log(1+\exp(\beta x)),
\end{align}
with $\beta = 10$. There are $L=4$ hidden layers, with a hidden dimension of $H=50$. We model our MMSE denoiser as the gradient of the DNN, i.e., $\psi_\theta \coloneqq \nabla g_\theta$. To learn the denoiser, we optimize the objective
\begin{align}\label{eq:MSE_loss}
\mathcal{L}(\theta)\coloneqq \frac{1}{B}\sum_{i=1}^B \|\psi_{\theta}(z_i) - x_i\|^2,\end{align}
with a batch size $B=2000$, $x_i \sim p_X$, and $z_i = x_i + \epsilon_i$, $\epsilon_i \sim \mathcal{N}(0,1)$. We consider the cases where $p_X$ is either a mixture-of-Gaussian or mixture-of-Laplacian distribution. We trained the DNNs using the Adam optimizer for $500$ epochs with a learning rate $10^{-3}$ followed by $500$ epochs with a learning rate of $10^{-4}$.

\begin{figure}[!htbp]
    \centering
    \includegraphics[width  = 0.9\linewidth]{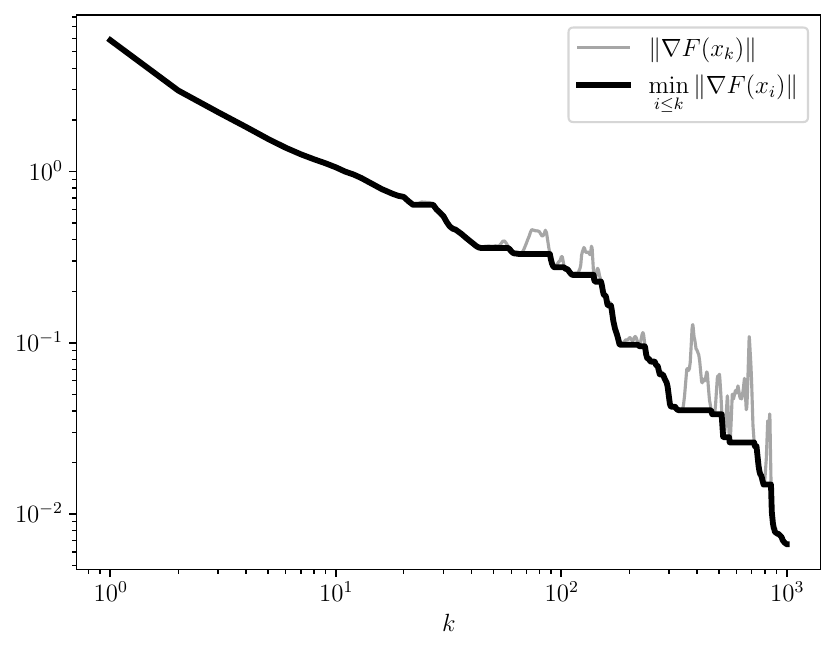}
    \caption{Best-iterate gradient norm over $\num{1000}$ iterations of PnP-PGD with an MMSE denoiser for the MayoCT computed tomography experiment in \cref{section:CT}, plotted on a log-log scale. The quantity decays sublinearly, consistent with a $\mathcal{O}(1/\sqrt{K})$ rate.}
    \label{fig:1.6}
\end{figure}

As discussed in \cref{sec:main}, the MMSE denoiser $\psi_\sigma$ can be written as $\operatorname{prox}_{\phi_{\mathrm{MMSE}}}$, where $\phi_{\mathrm{MMSE}}$ is the implicit regularizer. To approximate this, we take advantage of~\cite[Eq.~(8)]{gribonval_characterization} and write
\begin{align} \label{eq:phi_mmse_est}
    \phi^{\mathrm{learned}}_{\mathrm{MMSE}}(x) \coloneqq \langle \psi_\theta^{-1}(x),x\rangle - g_\theta(\psi_{\theta}^{-1}(x)) - \tfrac{1}{2}x^2,
\end{align}
where we compute $\psi_\theta^{-1}$ numerically via least-squares. 
When $\psi_\theta$ is the \textit{bona fide} MMSE denoiser, $\phi^{\mathrm{learned}}_{\mathrm{MMSE}}$ is the \textit{exact} implicit regularizer~\cite{gribonval_characterization} up to a constant.

\begin{figure*}[!t]
    \centering
    \includegraphics[width=\linewidth]{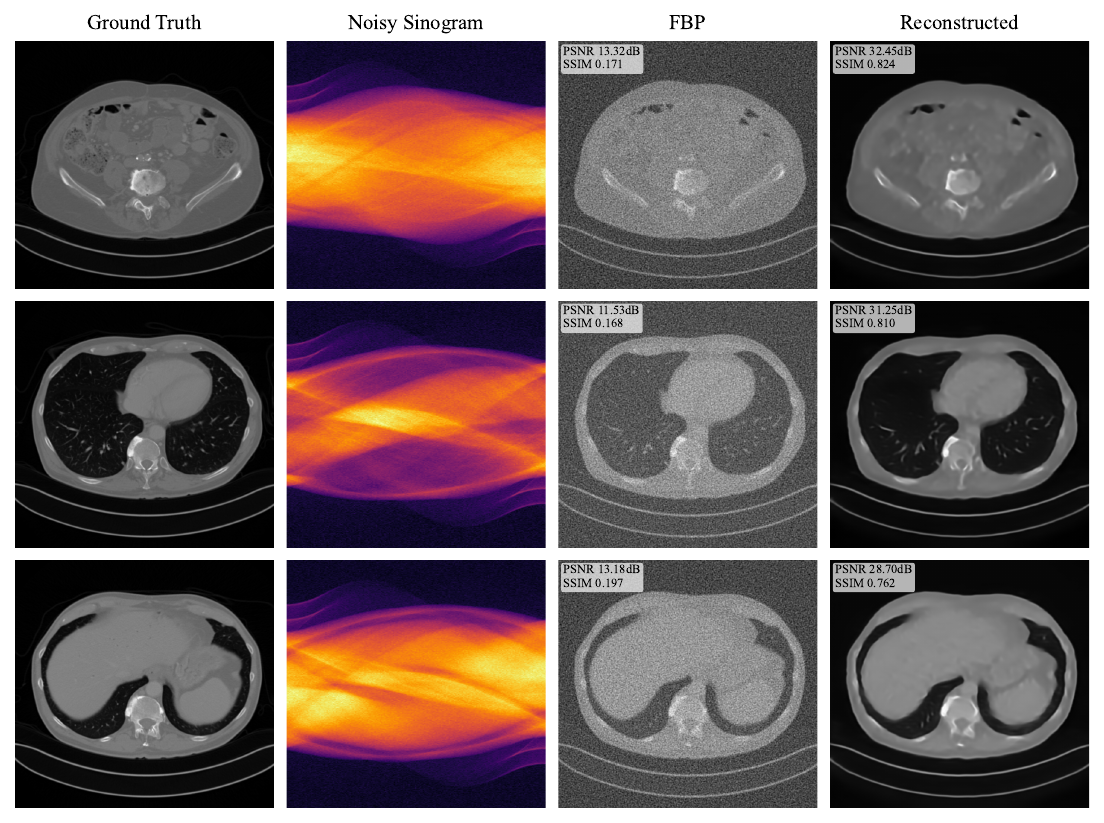}
    \caption{PnP-PGD reconstruction results on three axial slices from the MayoCT dataset. From left to right: clean reference, noisy sinogram, filtered backprojection of the noisy sinogram, and PnP-PGD output using an MMSE denoiser with early stopping. PSNR and SSIM are reported for the FBP and reconstructed images. Reconstructions achieve PSNR gains of roughly 15–20 dB over the FBP baseline, with SSIM improvements around $0.65$. 
}
    \label{fig:1.7}
\end{figure*}

Finally, since we can find the \emph{true} MMSE denoiser $\psi_{\sigma}$ via Tweedie's formula (\cref{thm:tweedies}), we compute the corresponding implicit regularizer $\phi_{\mathrm{MMSE}}$. From \cref{eq:moreau_and_prox} and \cref{lem:moreau_identity}, we have the closed-form expression
\begin{align}
    \phi_{\mathrm{MMSE}}(x) &= f_Z(\psi_{\sigma}^{-1}(x)) \nonumber \\
    &\qquad\qquad -\tfrac{1}{2}(x-\psi^{-1}_{\sigma} (x))^2-C_{x_0}, \label{eq:explicit_calculation}
\end{align}
since $\sigma^2 = 1$ in this setting. Again, $\psi_{\sigma}$ is inverted numerically via least-squares.

In \cref{fig:1.1,fig:1.2} we compare the learned vs.\ calculated implicit regularizers for three mixture-of-Gaussian and three mixture-of-Laplacian priors, along with the reference curves $f_X \coloneqq- \log p_X$ and $f_Z \coloneqq -\log\left(p_X*\mathcal{N}(0,1) \right)$. Since the purpose of this experiment is visual inspection, we ignore the constant $C_{x_0}$ in \cref{lem:moreau_identity} and shift the curves vertically so they can be easily compared.

\subsection{DNN Architecture for Inverse Problems}
\label{section:architecture}

We now detail the DNN architecture for the inverse problems experiments that appear in \cref{section:deblurring} (Gaussian deblurring) and \cref{section:CT} (Computed Tomography). In both cases, we define $h_1 \coloneqq \sigma(W_1x)$, and let 
\begin{align}
    h_{k} \coloneqq \sigma(W_{k}h_{k-1} + A_{k-1} \mathcal{I}_k(x) + b_k),
\end{align}
where $k = 2, \ldots, (L-1)$ and $\mathcal{I}_k(\cdot)$ denotes bilinear interpolation of $x$ so that the dimensions match for matrix-vector multiplication. We let the scalar output be $g_\theta(x)$, defined as
\begin{align}
    w_{\mathrm{out}}^\T \sigma(W_L \operatorname{vec}(h_{L-1}) + A_L\operatorname{vec}(\mathcal{I}_L(x)) +  b_L).
\end{align}
The DNNs use a softplus activation \cref{eq:softplus} with $\beta = 10$. Each layer is either a convolutional or fully connected layer. There is a skip connection at every layer except the first and last, and we provide the exact architectural details in \cref{tab:architecture}. We train with the mean-squared error loss \cref{eq:MSE_loss} with $z_i = x_i + \epsilon_i$ for zero-mean Gaussian noise $\epsilon_i$ with variance $\sigma^2 = 0.04$.

\subsection{Gaussian Deblurring}\label{section:deblurring}
We next consider the Gaussian deblurring inverse problem \cref{eq:inverse_problem}, where $x$ is randomly sampled from the MNIST dataset of handwritten digits~\cite{lecun_mnist}. The forward model $A$ is a Gaussian blurring operator, computed by circular convolution with a $(3\times 3)$ Gaussian kernel with $\sigma^2=1$. The kernel is normalized to have unit mass. We normalize the blur operator to have operator norm less than 1 to satisfy the assumptions in \cref{thm:PGD_MMSE_convergence}.

We use the DNN architecture described in \cref{section:architecture},  trained using the Adam optimizer with a cosine annealing learning rate schedule. The network is trained with batches of $500$ for $\num{50000}$ epochs with a final learning rate of $3\times10^{-5}$. We let $y=Ax+\epsilon$, with $\epsilon$ zero-mean Gaussian noise with variance $\sigma^2 = 0.04$. We use the data-fidelity term $f(x) \coloneqq\tfrac{1}{2}\|Ax-y\|^2$, and plot $\min_{i \leq K}\|\nabla F(x_i)\|$ for the first $100$ iterations in \cref{fig:1.4}, which we computed via \cref{eq:gradient_calculation}.  We display the reconstruction results in \cref{fig:1.3}. Observe that there is clear sublinear decay, consistent with our theoretical results in \cref{thm:PGD_MMSE_convergence} and \cref{cor:conv}.

\subsection{Computed Tomography} \label{section:CT}
We now consider the computed tomography inverse problem. The forward operator $A$ is given by a parallel beam Radon transform discretized via the Operator Discretization Library~(ODL)~\cite{adler_operator}. We scale the operator by $1/(1+\|A\|_{\mathrm{op}})$ in accordance with \cref{thm:PGD_MMSE_convergence}. We use 200 projection angles, uniformly spaced over $[0,\pi)$, with 400 detector bins, implemented using the Astra toolbox~\cite{van_fast} with CUDA acceleration. 

We use the DNN architecture described in \cref{section:architecture},  trained using the Adam optimizer with a cosine annealing learning rate schedule. The network is trained with batches of $64$ for $\num{40000}$ epochs with a final learning rate of $10^{-6}$. We use training and testing data from the MayoCT dataset~\cite{mcollough_low} (resolution of $512\times512$, pixel values ranging from 0 to 1). We let $y=Ax+\epsilon$, with $\epsilon$ zero-mean Gaussian noise with variance $\sigma^2 = 0.04$. We use the data-fidelity term $f(x) \coloneqq\tfrac{1}{2}\|Ax-y\|^2$ and plot $\min_{i \leq K}\|\nabla F(x_i)\|$ for the first $\num{1000}$ iterations in \cref{fig:1.6}, computed via \cref{eq:gradient_calculation}.  We display the reconstruction results for three slices of the MayoCT dataset in \cref{fig:1.7}.  There is clear sublinear decay, consistent with our theoretical result in \cref{thm:PGD_MMSE_convergence} and \cref{cor:conv}.

\section{Conclusion}
We presented a novel representation of the MMSE denoiser’s implicit regularizer in \cref{thm:mmse_prior} and used it to derive nonasymptotic convergence results for PnP-PGD using an MMSE denoiser in \cref{thm:PGD_MMSE_convergence} and \cref{cor:conv}. Our analysis is strictly tied to the Gaussian-noise MMSE setting.  Tweedie’s formula, for example, is not the same under other noise models~\cite{efron_tweedies}, so we suspect this flavor of result is unique to the Gaussian setting. However, we expect analogous nonasymptotic guarantees to be attainable for other PnP schemes that use MMSE denoisers under Gaussian noise, such as PnP-ADMM or the proximal point method. Exploration of these extensions constitutes future work.

\section*{Acknowledgment}
The authors would like to thank Stanislas Ducotterd and Guillaume Lauga for helpful feedback on this work.

\bibliographystyle{IEEEtranS} 
\bibliography{ref}

\appendices

\end{document}